\documentclass[11pt]{amsart}
\usepackage{amssymb, latexsym, amscd, amsmath}
\usepackage[mathscr]{eucal}
\usepackage{enumitem}
\usepackage[hmargin=0.75in,vmargin=0.75in]{geometry}

\usepackage[pdftex]{graphicx, color}

\usepackage[roman]{sublabel}

\usepackage[labelfont=bf]{caption}
\usepackage{subfig}

\definecolor{link_red}{rgb}{0.7,0,0}
\definecolor{cite_blue}{rgb}{0,0,0.97}

\usepackage{color}

 \usepackage{macros}

\theoremstyle{plain}
\newtheorem{theorem}{Theorem}[section]

\newtheorem{proposition}[theorem]{Proposition}

\theoremstyle{definition}
\newtheorem{definition}[theorem]{Definition}

\begin{document}

\title{Population Models with Partial Migration}

\author{Anushaya Mohapatra}
\address{Department of Mathematics \\
Oregon State University \\
Corvallis, OR 97330}
\urladdr[Anushaya Mohapatra]{}

\author{Haley A. Ohms}
\address{Department of Integrative Biology \\
Oregon State University \\
Corvallis, OR 97330}
\urladdr[Haley Ohms]{}

\author{David A. Lytle}
\address{Department of Integrative Biology \\
Oregon State University \\
Corvallis, OR 97330}
\urladdr[Haley Ohms]{}

\author{Patrick De Leenheer}
\address{Department of Mathematics \\
Oregon State University \\
Corvallis, OR 97330}
\urladdr[Patrick De Leenheer]{}


\subjclass[2000]{37}

\date{\today}

\begin{abstract}
Populations exhibiting partial migration consist of two groups of individuals: Those that migrate between habitats, and those that remain 
fixed in a single habitat. We propose several discrete-time population models to investigate the coexistence of migrants and residents. 
The first class of models is linear, and we distinguish two scenarios. In the first, there is a single egg pool to which both populations contribute. A fraction of the eggs is destined to become migrants, and the remainder become residents.
In a second model, there are two distinct egg pools to which the two types contribute, one corresponding to residents and another to migrants. The asymptotic growth or decline in these models can be phrased in terms of the value of the basic reproduction number being larger or less than one respectively.
A second class of models incorporates density dependence effects. It is assumed that increased densities in the various 
life history stages adversely affect the success of transitioning of individuals to subsequent stages. Here too we consider models with one or two egg pools. Although these are nonlinear models, their asymptotic dynamics can still be classified in terms of the value of a locally defined basic reproduction number: If it is less than one, then 
the entire population goes extinct, whereas it settles at a unique fixed point consisting of a mixture of 
residents and migrants, when it is larger than one. Thus, the value of the basic reproduction number can be used to predict the 
stable coexistence or collapse of populations exhibiting partial migration.

\end{abstract}

\maketitle
\tableofcontents

\section{Introduction}
\label{s:intro}

The phenomenon of partial migration, where a population is composed of a mixture of individuals that migrate between habitats and others that remain resident in a single habitat  \cite{lack_problem_1943}, has received extensive attention from biologists seeking to understand its origin and maintenance \cite{chapman_ecology_2011}. Examples of partial migration are diverse and come from nearly every class of organisms. Classic examples include salmonid fishes (\emph{Oncorhynchus}, \emph{Salmo}, and \emph{Salvelinus}) that breed in streams, but contain some individuals that migrate to an ocean or lake and others that complete their entire life cycle in the stream \cite{dodson_evolutionary_2013}, house finches (\emph{Haemorhous mexicanus}) that share a common breeding site, but some fraction of individuals will migrate away from this site to overwinter \cite{belthoff_partial_1991}, and red-spotted newts (\emph{Notophthalmus viridescens}) that breed in a common pond, but a fraction will migrate to the forest to overwinter \cite{grayson_life_2011}.

In many cases, migratory and resident individuals can interbreed and produce offspring that can become either migratory or resident \cite{chapman_ecology_2011}. Given the unlikely scenario that migrant and resident forms have equal fitness, one type will have a higher fitness than the other and should "win" the evolutionary competition. Yet, despite this evolutionary intuition, partial migration is maintained in a surprisingly large number of organisms under a wide range of environmental circumstances \cite{chapman_ecology_2011}.

	A number of mathematical models have been developed to explain the evolution, maintenance, and resulting dynamics of partial migration, and these depend on assumptions such as frequency dependence, density dependence, and condition dependence (\cite{kokko_modelling_2007, lundberg_evolutionary_2013, kaitala_theory_1993, taylor_predicting_2007, hazel_polygenic_1990, hazel_modeling}). 
	Frequency dependence occurs when a vital rate (survivorship, fecundity, egg allocation strategy) depends on the proportion of individuals in the population adopting a particular strategy. For example, survivorship of migrants might be higher when migrants are rare relative to residents. Density dependence occurs when a vital rate depends on the total number of individuals within a stage or class, resulting in a non-linear relationship between numbers of individuals and population growth. 
Condition dependence occurs when external factors such as temperature or food availability influence the decision to migrate or not. In particular, density dependence has been identified as an important factor for coexistence of migrants and residents, but it has only been explored in the context of frequency dependence \cite{kaitala_theory_1993, kokko_dispersal_2001}. For example, models that use an evolutionarily stable strategy approach, such as Kaitala and others (\cite{kaitala_theory_1993}), inherently involve frequency dependence among strategy types, in addition to other factors of interest, such as density dependence or condition dependence.  Our goal in this paper is to explore the circumstances under which density dependence alone can lead to the stable coexistence of resident and migratory forms.

	We based our general model structure on the life history of the salmonid fish \emph{Oncorhynchus mykiss}, which expresses an ocean migratory form (steelhead) as well as a freshwater resident form (rainbow trout). This is an example of 'non-breeding partial migration' \cite{chapman_ecology_2011}, in which migrants and residents interbreed in a common habitat. Steelhead and rainbow trout spawn in freshwater streams and their young rear in these streams from one to three years \cite{busby_status_1996}. After this period some individuals will become steelhead and migrate to the ocean and others will remain in the stream as rainbow trout. Using a Leslie matrix model framework based on \emph{O. mykiss}, we assessed conditions for coexistence that do not rely on frequency dependence, but instead rely only on stage-specific density-dependence. Although the model is based on \emph{O. mykiss}, we use the terminology ‘migrants’ and ‘residents’ throughout this paper to highlight the applicability of this model framework to many partially migratory species.  We prove that a locally defined fitness quantity ($R_0$) determines the global dynamics of the system, and that there exist stable equilibrium points that allow the coexistence of both resident and migratory forms.

\section{Linear population models}
\label{s:l_stat_res}
We start with a linear stage-structured population model, representing either an isolated population of only migrants or residents with $n$ stages. This is a single-sex model where we are only modeling females.
Later we will consider various coupled population models of migrants and residents.  

Let $x_i$ denote the number of (female) individuals in stage $i$, where $i$ ranges from $1$ to $n$. We let $t_i$ be the expected transition probabilities (so that $0\leq t_i\leq 1$) for individuals from stage $i$ to stage $i+1$ for $i=0,\dots, n-1$, and let $t_n$ denote the survival probability of individuals in the final $n$th stage. Finally, we let $f_i$ be the fecundities (number of eggs) for $i=2,\dots,n$. They represent the expected number of offspring produced by an individual in the $i$th stage.

The dynamics of the number of individuals in the various stages is described by the following Leslie model:
\begin{equation}
\label{model}
\mbi{x}(t+1)= A_{n} \mbi{x}(t),
\end{equation}\\
with
$A_{n}=\left( \begin{array}{cccc}
                                      0 & f_{2} & \cdots & f_{n} \\
                                      t_{1} & 0 & \cdots & 0\\
                               
                                     \vdots & \ddots  &  \vdots & \vdots\\
                                      0 &  \cdots  & t_{n-1} & t_{n}\end{array}\right)$
and  $x=\left( \begin{array}{c}
                                      x_{1}\\
                                       \vdots \\
                                       x_{n}                            
                                     \end{array}\right).$\\
 
We assume that the matrix $A_n$ is irreducible, or equivalently that all $t_i>0$ for $i=1,\dots,n-1$, and that $f_n>0$ \cite{caswell_matrix_2000}. The asymptotic behavior of the non-negative solutions of ~\eqref{model} is determined by the eigenvalue $\lambda$ of $A_n$ that has largest modulus. Since $A_n$ is a matrix whose entries are non-negative, the Perron-Frobenius theorem implies that $\lambda$ 
 is in fact real and non-negative. Consequently, when $\lambda<1$, all solutions converge to zero, and when $\lambda>1$, then all solutions 
 eventually grow at exponential rate $\lambda$. Unfortunately, calculating $\lambda$ in terms 
 of the entries of the matrix $A_n$ is generally not possible. However, an associated quantity, known as the \textbf{basic reproduction number} $R_{0}$, can be determined explicitly. Letting $\rho(M)$ be the spectral radius of any square matrix $M$, the basic 
 reproduction number $R_0$ is defined as \cite{Population,li-schneider,allen}:
$$
R_{0}:= \rho \left(F_{n}(I-T_{n})^{-1}\right),
$$ 
where the matrices $F_{n}$ and $T_{n}$ are obtained from $A_{n}$ by setting all $t_{i}=0$ and  all $f_{i}=0$ respectively. A straightforward calculation shows that
\begin{equation}\label{R0-explicit}
R_0=t_1f_2+t_1t_2f_3+\dots+(t_1\dots t_{n-2})f_{n-1}+\frac{t_1\dots t_{n-1}}{1-t_n}f_n,
\end{equation}
and reveals the interesting biological interpretation for $R_0$ as {\it the expected total number of offspring contributed to stage $1$, generated by a single stage $1$ individual over its lifetime}. Although this statement is often taken as the definition of the basic reproduction number in the literature, we 
remark that this is only possible because the model has only one stage in which new offspring is generated. When there are several stages 
in which offspring are generated (we will later consider models for which this is the case), this clear biological interpretation for the 
basic reproduction number is no longer possible. 

There is an important connection between $\lambda$ and $R_0$, see \cite{li-schneider}: Exactly one of the $3$ following scenarios occurs:
\begin{equation}\label{relation}
\begin{cases}
0\leq R_0\leq \lambda <1,\textrm{ or }\\
\lambda=R_0=1,\textrm{ or}\\
1<\lambda \leq R_0.
\end{cases}
\end{equation}
Consequently, the long-term behavior of solutions of $(\ref{model})$ can be predicted based on whether $R_0$ is less than, or larger than $1$: If $R_0<1$, then all solutions converge to zero, and if $R_0>1$, then all solutions eventually grow exponentially. 
In other words, it does not matter whether we use the value of $\lambda$, or the value of $R_0$ to determine population growth or extinction: the location of either one with respect to the threshold $1$, is sufficient to decide this issue. Of course, the fact that $R_0$ is usually more easily computed than $\lambda$, makes it the natural candidate to address this question, and this probably explains why $R_0$ appears to be a more popular measure than $\lambda$ in the mathematical population biology literature. 

A final remark concerns the sensitivity of both $\lambda$ and $R_0$ with respect to the natural demographic model parameters 
$t_i$ and $f_i$: Both quantities are increasing with respect to increases in any of these parameters. Indeed, for $\lambda$ this property is a consequence of the Perron-Frobenius Theorem which says that the spectral radius of a non-negative, irreducible matrix, increases with 
any of the matrix entries. For $R_0$, this property is immediately clear from $(\ref{R0-explicit})$. Thus, from the perspective of a 
sensitivity analysis, it also does not matter which of the two quantities, $\lambda$ or $R_0$, we consider.

Whereas model $(\ref{model})$ represents an isolated population of either migrants or residents, in practice these populations 
are coupled. In the next few sections we explore how various assumptions about how exactly these populations are 
coupled affect the long term dynamics of the coupled populations.

\subsection{A coupled model with a single egg stage} 
\label{C:M1}
Here we consider the case where the population is composed of migrants and residents that have $n$ and $m$ stages respectively. The potentially different values $n$ and $m$  reflect the 
fact that migrants and residents will sometimes reproduce (i.e., spawn) at different ages. In the case of steelhead and rainbow trout, steelhead will migrate to the ocean and remain there for extended periods of time to grow before returning as adults, whereas rainbow trout and reach adulthood more quickly \cite{busby_status_1996}. 

The key feature of this model is the assumption that there is a single pool of eggs to which both spawning migrants and residents contribute. This would be the biological case where an event occurs between the egg and juvenile stage that determines an individual's life history trajectory (i.e., whether they become resident or migrant). This could be a threshold growth value \cite{dodson_evolutionary_2013}, a threshold lipid value \cite{sloat_individual_2014}, or a proximiate cue such as density \cite{tomkins_population_2004}. 
A fraction $0<\phi<1$ of the eggs become migrants, and the remaining fraction $1-\phi$ become residents. As before, the transition probabilities between the migrant stages are denoted by $t_i^s$ with $i=1,\dots,n$ and those between the resident stages by $t_j^r$ with $j=1,\dots,m$. Similarly, the migrant fecundities are $f_i^s$ with $i=2,\dots,n$, and the resident fecundities are $f_j^r$ with $j=1,\dots,m$. We also continue to assume that all $t_i^s$, and $t_j^r$, and all  $f_n^s$ and $f_m^r$ are positive. The coupled model takes the following form:

\begin{equation}
\label{model:1c}
\mbi{X}(t+1)= A_1\mbi{X}(t),
\end{equation}  
where
$A_1=\left( \begin{array}{cccccccc}
                                      0 & f_{2}^s & \cdots & f_{n}^s&f_2^r &\cdots & f_{m-1}^r&f_m^r\\
                                      \phi t_{1}^s & 0 & \cdots & 0&0&\cdots& 0&0 \\
                                      \vdots & \ddots  &  \vdots & \vdots &\vdots&\cdots&\vdots&\vdots\\
                                      0 &  \cdots  & t_{n-1}^s & t_{n}^s&0&\cdots&0&0\\
                                      (1-\phi)t_1^r&0&\cdots&0&0&\cdots&0&0\\
                                      0&0&\cdots&0&t_2^r&\cdots & 0&0\\
                                      \vdots&\vdots&\cdots&\vdots&\vdots& \ddots&\vdots& \vdots \\
                                      0&0&\cdots&0&0&\cdots & t_{m-1}^r& t_m^r
                                      \end{array}\right)$,\;
$\mbi{X}=\left(
\begin{array}{c}
\text{eggs}(x_1)   \\
\text{first migrant stage} (x_2)  \\
\vdots \\
\text{last migrant stage}  (x_n)  \\
\text{first resident stage}  (x_{n+1})  \\
\vdots \\
\text{last resident stage}  (x_{n+m-1})  \\
\end{array}
\right)$\\ 
\vspace{0.2in}

We can associate the basic reproduction number 
to the coupled system ~\eqref{model:1c}, which is defined in the familiar way:
$$
R_0^{c1}:=\rho \left(F_1(I-T_1)^{-1} \right),
$$ 
where $F_1$ and $T_1$ are obtained from $A_1$ by setting all transition probabilities $t_i^s$ and $t_j^r$, respectively all fecundities $f_i^s$ and $f_j^r$, equal to zero. Since $0<\phi<1$, the matrix $A_1$ is a non-negative irreducible matrix having a real, nonnegative eigenvalue $\lambda_1$ of maximal modulus, and all the remarks about the relationship between $\lambda_1$ and $R_0^{c1}$, mentioned for model $(\ref{model})$, are also valid here. In particular, $(\ref{relation})$ holds when replacing $R_0$ by $R_0^{c1}$ and $\lambda$ by $\lambda_1$, and therefore population growth or extinction is determined by the location of either $\lambda_1$ or $R_0^{c1}$ with respect to the threshold $1$. 
Another important feature of $R_0^{c1}$ is that it has the usual biological meaning of the expected total number of eggs, generated by a single egg over its lifetime.  
The remarks about the sensitivity of $R_0^{c1}$ with respect to any of the model's transition probabilities and fecundities continues to hold as well.

The sensitivity of $R_0^{c1}$ with respect to the allocation parameter $\phi$ is not immediately obvious. Indeed, one of the entries in $A_1$'s first column increases with $\phi$, whereas another decreases. However, 
it turns out that there is a particularly elegant formula for $R_0^{c1}$ in terms of the basic 
reproduction numbers associated to an isolated $n$-stage migrants model and an isolated $m$-stage resident model.  
To make this precise, we let $R_0^s$ denote the basic reproduction number of model $(\ref{model})$ with $A_n^s$ instead of $A_n$, in which we replace the $f_i$ by $f_i^s$, and the $t_i$ by $t_i^s$. Similarly, we let $R_0^r$ be the basic reproduction number of model $(\ref{model})$ with $A_m^r$ (i.e. there are $m$ instead of $n$ stages), and replace the $f_j$ by $f_j^r$, and the $t_j$ by $t_j^r$. With this notation, it is not difficult to show that $R_0^{c1}$ is a convex combination of $R_0^s$ and $R_0^r$ with weights $\phi$ and $1-\phi $ respectively:
 \begin{equation}\label{convex}
 R_{0}^{c1}(\phi)= \phi R_{0}^{s} + (1-\phi) R_{0}^{r}
 \end{equation}
 
Clearly, $R_0^{c1}(\phi)$ is linear function of the allocation parameter $\phi$, and consequently, $R_0^{c1}(\phi)$ is always between $R_0^s$ and $R_0^r$. The maximal and minimal values of $R_0^{c1}(\phi)$ are $\max(R_0^s,R_0^r)$ and $\min(R_0^s,R_0^r)$ respectively, and each is achieved for a single, extreme value of $\phi$ -namely $\phi=0$ or $1$- unless 
 $R_0^s=R_0^r$, in which case $R_0^{c1}(\phi)$ is independent of $\phi$: $R_0^{c1}(\phi)\equiv R_0^s$ for all $\phi$ in $[0,1]$.
 
 These observations imply that generically (i.e. when $R_0^s\neq R_0^r$), there is no value of the allocation parameter $\phi$ that 
 gives rise to a basic reproduction number $R_0^{c1}$ of the coupled model $(\ref{model:1c})$ which is higher than both basic reproduction numbers $R_0^s$ and $R_0^r$, associated with the models of the isolated migrants and isolated residents respectively. 
 In fact, the maximal $R_0^{c1}$ is achieved for an extreme value of the allocation parameter: when $R_0^s>R_0^r$,  the maximum is $R_0^s$, and it is achieved when $\phi=1$, which corresponds to the scenario in which all eggs become migrants. Similarly, when $R_0^s<R_0^r$,  the maximum is $R_0^r$, and it is achieved when $\phi=0$, which corresponds to the scenario in which all eggs become residents.

\subsection{A coupled model with two distinct egg stages}  
\label{C:2}

The key feature of this model is the assumption that there are two pools of eggs: one pool will become migrants and the other will become residents. This would be the biological case where life history (i.e., migrant or resident) is determined at birth. Spawning adult migrants and residents contribute offspring to each pool of eggs. We let $0<\phi_{s}<1$ be the expected fraction of eggs generated by migrants that enter the migrant egg stage. Hence, $1-\phi_s$ represents the fraction of eggs generated by migrants that enter the resident egg stage. Similarly,  $\phi_{r}$ is the expected fraction of eggs generated by residents that enter the resident stage.  
 
We let $A_{n}^{s}$ and $A_{m}^{r}$  be the system matrices of an isolated $n$-stage migrant and isolated $m$-stage resident model, as defined in the previous section, and denote the respective basic reproduction numbers by $R_0^s$ and $R_0^r$ respectively. 
The coupled model then takes the form:
\begin{equation}
\label{model:2}
\mbi{X}(t+1)= A_2\mbi{X}(t)
\end{equation}\\
with

$A_2  = \begin{pmatrix} 
 \phi_{s}*A_{n}^{s} & (1-\phi_{r})* B_{n\times m}^{r}\\
 (1-\phi_{s})* B_{m\times n}^{s} &  \phi_{r}*A_{m}^{r} \\
 \end{pmatrix},\;\; \mbi{X}=\left(
\begin{array}{c}
\text{eggs that become migrants}(x_1)   \\
\text{first migrant stage} (x_2)  \\
\vdots \\
\text{last migrant stage}  (x_n)  \\
\text{eggs that become residents}  (x_{n+1})  \\
\vdots \\
\text{last resident stage}  (x_{n+m})  \\
\end{array}
\right),\vspace{0.2in}$

where $a * A$  is the matrix obtained from $A$ by multiplying all entries of the first row of $A$ by the scalar $a$, and not changing the entries of any of the other rows of $A$. The matrices $B_{n\times m}^r$ and $B_{m\times n}^s$ are rectangular with $n$ rows and 
$m$ columns, respectively $m$ rows and $n$ columns. The first row of the matrix $B_{n\times m}^r$ is  the same as the first row of the matrix $A_m^r$, and all other rows of $B_{n\times m}^r$ consist of zeros. The matrix $B_{m\times n}^s$ is constructed in a similar way from the matrix $A_n^s$. 

Since $A_2$ is an irreducible, non-negative matrix, we can associate the basic reproduction number to model $(\ref{model:2})$, which is defined as follows:
$$
R_0^{c2}:=\rho \left(F_2(I-T_2)^{-1} \right),
$$ 
where $F_2$ and $T_2$ are obtained from $A_2$ by setting all transition probabilities $t_i^s$ and $t_j^r$, respectively all fecundities $f_i^s$ and $f_j^r$, equal to zero. The matrix $A_2$ has a real, non-negative eigenvalue $\lambda_2$ of largest modulus, and the 
comments in the previous subsection regarding the relationship between $\lambda_1$ and $R_0^{c1}$ carry over to the relationship between  $\lambda_2$ and $R_0^{c2}$. The remarks on the sensitivity of  $\lambda_2$ and $R_0^{c2}$ with respect to the transition probabilities and fecundities carry over as well.

There is one notable difference however, namely that unlike $R_0^{c1}$, the basic reproduction number $R_0^{c2}$ can not be interpreted biologically as some expected number of eggs generated by a single egg over its lifetime, because here there are two distinct egg stages. 

To exhibit the dependence of $R_0^{c2}$ on the allocation 
parameters $\phi_s$ and $\phi_r$, we note that a straightforward calculation shows that:
\begin{eqnarray*}
\label{eq:1}
R_{0}^{c2}(\phi_s,\phi_r) &= &\rho \begin{pmatrix}
\phi_s R_0^s& (1-\phi_r)R_0^r\\
(1-\phi_s)R_0^s&\phi_r R_0^r
\end{pmatrix}\\
&=&\frac{\phi_sR^s_0 + \phi_rR^r_0 + \sqrt{(\phi_sR^s_0 + \phi_rR^r_0)^2 - 4R^s_0R^r_0(\phi_s + \phi_r - 1)}}{2},
\end{eqnarray*}
which is, in general, not a convex function of $(\phi_s,\phi_r)$ in $D=\{(\phi_{s}, \phi_{r}): 0 \leq \phi_{r} \leq 1, 0 \leq \phi_{s} \leq 1\}$.
 
 \begin{proposition}
 \label{th:1}
 If  $R^{s}_{0}=R^{r}_{0}$, then $R_{0}^{c2}(\phi_{s}, \phi_{r})$ is  a constant function with value  $R^{s}_{0}$.\\
If $R_0^s\neq R_0^r$, then $R_{0}^{c2}(\phi_{s}, \phi_{r})$  has no interior maxima or minima on $D$.  
Moreover, for  $R^{r}_{0} < R^{s}_{0}$, $R_{0}^{c2}(\phi_{s}, \phi_{r})$ attains its maximum $R^{s}_{0}$ on any point of the boundary of $D$ where $\phi_{s}=1$.  For  $R^{s}_{0} < R^{r}_{0}$, $R_{0}^{c2}(\phi_{s}, \phi_{r})$ attains its maximum $ R^{r}_{0}$ on any point of the boundary of $D$ where $\phi_{r}=1$.
\end{proposition}
\begin{proof}
\textbf{Case 1}: $R_0^s= R_0^r$. In this case, a straightforward calculation shows that $R_0^{c2}(\phi_s,\phi_r)\equiv R_0^s$.\\

\textbf{Case 2}: $R_0^s\neq R_0^r$. Evaluating the partial derivatives of $R_0^{c2}$ with respect to $\phi_s$ and $\phi_r$ yields that

\begin{eqnarray*}
2\frac{\partial R_{0}^{c2}}{\partial \phi_{s}} &=&  R^{s}_{0}+ \frac{(R^{s}_{0} \phi_{s}+ R^{r}_{0} \phi_{r}) R^{s}_{0}-2 R^{s}_{0} R_{0}^{r}}{\sqrt{(\phi_{s}R^{s}_{0} + \phi_{r} R^{r}_{0})^2 - 4R^s_0R^r_0(\phi_{s} + \phi_{r} - 1)}}\\
2\frac{\partial R_{0}^{c2}}{\partial \phi_{r}} &=&  R^{r}_{0}+ \frac{(R^{s}_{0} \phi_{s}+ R^{r}_{0} \phi_{r}) R^{r}_{0}-2 R^{s}_{0} R_{0}^{r}}{\sqrt{(\phi_{s}R^{s}_{0} + \phi_{r} R^{r}_{0})^2 - 4R^s_0R^r_0(\phi_{s} + \phi_{r} - 1)}}
\end{eqnarray*}

Setting $\frac{\partial R_{0}^{c2}}{\partial \phi_{r}}=\frac{\partial R_{0}^{c2}}{\partial \phi_{r}}=0$ and simplifying leads to $R^{s}_{0}= R^{r}_{0}$.  Hence, since by assumption $R_0^s\neq R_0^r$, there cannot be any maxima or minima  in the interior of $D$.\\ 
If $R^{r}_{0} < R^{s}_{0}$, then a simple calculation shows that:
\begin{equation}
\label{eq:6}
 R_{0}^{c2}(\phi_{s}, \phi_{r}) =
 \begin{cases}
R_0^s  & \text{if $\phi_{s}=1$} \\
  \frac{\phi_{r}R^r_0 + \sqrt{(\phi_rR^r_0)^2 - 4R^s_0R^r_0(\phi_r -1)}}{2} & \text{ if $\phi_{s} =0$} \\
  \frac{\phi_{s}R^s_0 + \sqrt{(\phi_sR^s_0)^2 - 4R^s_0R^r_0(\phi_s -1)}}{2} & \text{ if $\phi_{r} =0$} \\
  \frac{R^s_0 \phi_{s}+R^r_0 + \abs{\phi_{s} R^s_0 - R^r_0}}{2} &   \text{if $\phi_{r}=1$}
\end{cases}
\end{equation}
 From equation ~\eqref{eq:6} we observe that the derivative of  the single variable functions $R_{0}^{c2}(\phi_{s}, 0)$, and $R_{0}^{c2}(0,  \phi_{r})$ are never zero on the interval $(0, 1)$.  The values at the boundary points of $(0,1)$ are given by: 
 \begin{equation}
\label{eq:7}
 R_{0}^{c2}(\phi_{s}, \phi_{r}) =
 \begin{cases}
  R_{0}^{s} & \text{if $\phi_{s}=1$} \\
   (R_{0}^sR_0^r)^{1/2} & \text{ if $\phi_{s} =0, \phi_{r}=0$} \\
  R_{0}^{r} & \text{if $\phi_{s} =0, \phi_{r}=1$}\\
  \end{cases}
\end{equation}
Hence the maximum $R_{0}^{s}$ is attained on $D$ when $\phi_{s}=1$.\\
If $R^{r}_{0} < R^{s}_{0}$, it can be shown in a similar way that the  maximum of $R_0^{c2}$ on $D$ equals $R_{0}^{r}$ and that it is attained when $\phi_{r}=1.$                                
\end{proof}

This result implies that generically (i.e. when $R_0^s\neq R_0^r$),  there is no allocation pair $(\phi_s,\phi_r)$ for which  
the basic reproduction number $R_0^{c2}(\phi_s,\phi_r)$ of model $(\ref{model:2})$ is larger than both reproduction numbers 
$R_0^s$ and $R_0^r$, associated to the models of the isolated migrants and isolated residents respectively. 
The maximal $R_0^{c2}$ is achieved for values of the allocation parameter on the boundary of $D$: 
when $R_0^s>R_0^r$,  the maximum is $R_0^s$, and it is achieved for $\phi_s=1$ and arbitrary values of $\phi_r$, which corresponds to the scenario in which all migrant eggs become migrants, regardless of the fraction of resident eggs that become residents. Similarly, when $R_0^s<R_0^r$,  the maximum is $R_0^r$, and it is achieved for $\phi_r=1$ and arbitrary values of $\phi_s$, which corresponds to the scenario in which all resident eggs become residents, no matter which fraction of migrant eggs become migrants. These conclusions are similar to those for the basic reproduction number $R_0^{c1}$ of model $(\ref{model})$.

\section{Nonlinear density-dependent models}
Thus far we have neglected any density-dependent effects. Here we assume that 
transition and survival probabilities depend on the density in each stage. This may be due to stage-specific competition for resources and/or space. 
We propose the following $n$-stage uncoupled model, which may be used to model an isolated migrant or an isolated resident population:
\begin{equation}
\label{model:1}
\mbi{x}(t+1)= A(\mbi{x}(t))\mbi{x}(t),
\end{equation}\\
where $\mbi{x}=\left( \begin{array}{ccc}
                                      x_{1} \\
                                      x_{2} \\
                                      \vdots\\
                                       x_{n}\end{array} \right), \; A(\mbi{x})=\left( \begin{array}{cccc}
                                      0 & f_{2} & \cdots & f_{n} \\
                                      t_{1}(x_{1}) & 0 & \cdots & 0\\
                               
                                     \vdots & \ddots  &  \vdots & \vdots\\
                                      0 &  \cdots  & t_{n-1}(x_{n-1}) & t_{n}(x_{n})\end{array}\right),$\\

 and assume that $t_{i}(x)$ for $1 \leq i \leq n$,  and  $f_{j}$  for $2 \leq j \leq n$ satisfy: 

\begin{enumerate}  [leftmargin=*, topsep=6pt, itemsep=4pt, label=(A\arabic*), ref=A\arabic*]

\item $t_i(x)$ is $C^1$ on $\mbb{R}_+$ with $ 0< t_{i}(x) \leq 1$ for all  $x \in \mbb{R}_{+} $, and $t_i$ is strictly decreasing. \label{A:1}
\item The function $s_i(x):=xt_{i}(x)$ is strictly increasing on $\mbb{R}_+$, and there is a bound $m_i \in \mbb{R}$ such that $s_{i}(x) \leq m_i$ 
for all $x$. \label{A:2}
\item $f_{j} \geq 0$, and $f_n>0$.
\label{A:3}
\end{enumerate}
A common choice for $t_i(x)$ is  $\frac{b_{i}}{1+c_{i}x}$ such that $s_{i}(x)$  is the Beverton-Holt function.  It satisfies (\ref{A:1})-(\ref{A:2}) 
when $0<b_{i}\leq 1$ and  $c_{i}>0$ with $m_i=b_i/c_i$, for $1\leq i \leq n$.\\

The linearization of  system ~\eqref{model:1} near the fixed point at the origin is: 
\begin{equation}
\label{linear:1}
\mbi{y}(t+1)= A_{0} \mbi{y}(t),
\end{equation}
where $A_{0}$ is the Jacobian of  $A(\mbi{x})\mbi{x}$  at the origin:  $A_0:=\left( \begin{array}{cccc}
                                      0 & f_{2} & \cdots & f_{n} \\
                                      a_{1} & 0 & \cdots & 0\\
                               
                                     \vdots & \ddots  &  \vdots & \vdots\\
                                      0 &  \cdots  & a_{n-1} & a_{n}\end{array}\right)$, where $a_i:=t_i(0)$ for all $i$.\\

By (\ref{A:1}) and (\ref{A:3}), $A_{0}$ is non-negative and irreducible.  Consequently, as we discussed in the previous sections, we can associate the basic reproduction number $R_{0}$ to the model. It is defined as $R_{0}:=\rho({F_{0}(I-T_{0})^{-1}})$ where $F_{0}$ and $T_{0}$ are obtained from $A_0$ by setting all $t_{i}=0$, and all $f_{i}=0$ respectively. Moreover, the largest eigenvalue in modulus $\lambda$ of $A_{0}$ and $R_{0}$  are either both greater than one, or both less than one, or both equal to one.\\

Although $R_{0}$ is a locally defined quantity, we will show that it determines the global dynamics of 
system ~\pref{model:1}. To prove this we shall apply the Cone Limit  Set Trichotomy of ~\cite{Monotone}. 
Before stating this important result, we need to introduce some definitions regarding monotone systems.
\subsection{Monotone systems}
We let ${\mbb R}^n_+$ be the non-negative cone in ${\mbb R}^n$. It consists of all vectors that have non-negative entries. 
For any $x$ and $y$ in ${\mbb R}^n$, we write $x\leq y$ ($x<y$) whenever $y-x \in {\mbb R}^n_+$ and when also ($x\neq y$). We 
write $x<<y$ if $y-x\in \textrm{Int}({\mbb R}^n_+)$.
\begin{definition}
A map $F: {\mbb{R}^{n}_+}\to {\mbb{R}^{n}_+}$ 
is said to be \textbf{monotone} if for $ x, y \in {\mbb{R}^{n}_+},    x \leq y $ implies that $F(x) \leq F(y)$.  
\end{definition}

\begin{definition}
A map $F: {\mbb{R}^{n}_+}\to {\mbb{R}^{n}_+}$ is \textbf{strongly sublinear} if 
$$ 0< \la < 1, \mbi{x} \gg 0 \Rightarrow \la F(\mbi{x}) \ll F(\la \mbi{x}).$$
\end{definition}

\begin{definition}
For  $x, y \in {\mbb{R}^{n}_+}$ with $x < y$,  we call $[x, y] := \set{z \in {\mathbb{R}^{n}}^{+}: x \leq z \leq y}$  the order interval with endpoints $x$ and $y$.  A subset of  ${\mathbb{R}^{n}_+}$ is \textbf{order bounded} if it is contained in some order interval.
\end{definition}

\begin{definition}
A map $F: {\mbb{R}^{n}_+}\to {\mbb{R}^{n}_+}$ is said to be \textbf{order compact} if takes every order bounded set into a precompact set.
\end{definition}
\begin{theorem}[\textbf{Cone Limit  Set Trichotomy}] 
\label{CLST}
 Assume  $ F: {\mbb{R}^{n}_+}\to {\mbb{R}^{n}_+} $ is continuous and monotone and has following properties for some integer $r\geq 1$, 
 where $F^r:=F\circ \dots \circ F$ denotes the $r$-fold composition of $F$ with itself:

\begin{enumerate} [leftmargin=*, topsep=6pt, itemsep=4pt, label=(H\arabic*), ref=H\arabic*]

\item  $F^{r}$ is strongly sublinear.
\label{H:1}
\item $F^{r}(x)  \gg 0$ for all $x > 0$.
\label{H:2}
\item $F^{r}$ is order compact.
\label{H:3}
\end{enumerate}

 Then precisely one of the following holds for the discrete dynamical system $\mbi{x}(t+1)=F(\mbi{x}(t))$:
 
 \begin{enumerate} [leftmargin=*, topsep=6pt, itemsep=4pt, label=(R\arabic*), ref=R\arabic*]

 \item each nonzero orbit is unbounded.
 \label{R:1}
 \item each orbit converges to $0$,  the unique fixed point of F.
 \label{R:2}
 \item each nonzero orbit converges to $q\gg 0$,  the unique non-zero fixed point of F.
 \label{R:3}
 \end{enumerate}
 
\end{theorem}

\subsection{A density-dependent model for an isolated population}
  \begin{theorem}
 \label{p:1}
 Assume that (\ref{A:1})-(\ref{A:3}) hold, and that for all $i=1,\dots,n$, the maps $s_i$ are strongly sublinear on ${\mbb R}_+$. 
If $R_{0} < 1$ then the zero fixed point is the only fixed point of the system \eqref{model:1} and it is globally asymptotically stable. If $R_{0} >1$, then system \eqref{model:1} has a unique locally stable positive fixed point which attracts all nonzero orbits.
\end{theorem} 
\begin{proof}: We will prove that the Cone Limit Set Trichotomy, Theorem ~\ref{CLST}, applies to the system \eqref{model:1}. 
Let $F(x(t)) = A(\mbi{x}(t)) \mbi{x}(t)$, and denote  $F(x)=(F_{1}(x), F_{2}(x), \cdots, F_{n}(x))$, it follows that  
$F_{1}(x)=\sum_{i=2}^{n} f_{i} x_{i}$, $F_{i}(x)= s_{i-1}(x_{i-1})$ for $i=2  \dots n-1$, and that $F_{n}(x)=s_{n-1}(x_{n-1}) + s_{n}(x_{n}).$ First, we will show that $F(x)$ satisfies the hypotheses ~\pref{H:1}  - ~\pref{H:3}. Each  $s_i$ is strictly increasing by (\ref{A:2}), and therefore $F$ is monotone on ${\mbb{R}^{n}_+}$. Clearly, $F$ is continuous is as well.

Remark that $F$ is not strongly sublinear because 
$\la F_{1}(x)-F_{1}({\la x})=0$ for $0< \la < 1$ and all $x$. Thus, we consider $F^{2}$, and we note that 
$F^{2}_{1}(x)=f_ns_n(x_n)+\sum_{i=2}^{n} f_{i} s_{i-1}(x_{i-1}), \; F^{2}_{2}(x)= s_{1}(\sum_{j=2}^nf_jx_j)$,  $F^{2}_{i}(x)= s_{i-1} \circ s_{i-2}(x_{i-2}), \, i=3, \dots, n-1$, and $F^{2}_{n}(x) = s_{n-1}\circ s_{n-2}(x_{n-2}) + s_{n}\circ(s_{n-1}(x_{n-1}) + s_{n}(x_n))$. 
It is easy to verify that the composition of any two functions $s_i$, is strongly sublinear because each $s_i$ is strictly increasing by (\ref{A:2}) and strongly sublinear by assumption. It follows that $ F^{2}$ is strongly sublinear because each coordinate function $F^{2}_i$ is  
a positive linear combination of strongly sublinear functions, hence strongly sublinear as well. A similar argument shows that 
$F^{r}$ is also strongly sublinear for any $r\geq 2$.

It follows from (\ref{A:1}) and (\ref{A:2}) that $s_i(x_i)>0$ for all $x_i>0$ and all $i=1,\dots, n$. This property and 
the fact that $f_n>0$ by (\ref{A:3}) imply that $F^n(x_0)>>0$ whenever $x_0>0$. 
  
Since $F^{n}$ is a continuous function on a finite dimensional space, it is order compact.
 
In summary, we have shown that $F$ satisfies hypotheses (\ref{H:1})-(\ref{H:3}) of the  Cone Limit Set Trichotomy for  $r=n$.

From the above calculation of the coordinate functions, it follows that $F^2$ maps ${\mbb R}^n_+$ into the order interval $[0,a]$, where 
$$
a:=\begin{pmatrix}
f_nm_n+\sum_{i=2}^nf_im_{i-1}\\
m_1\\
m_2\\
\vdots \\
m_{n-1}\\
m_{n-1}+m_n
\end{pmatrix}
$$
Since $F$ is continuous, $F([0,a])$ is compact, and therefore all orbits of system $(\ref{model:1})$ are bounded, hence also order bounded.  In particular, \pref{R:1} does not hold for system $(\ref{model:1})$. Therefore, either \pref{R:2} or \pref{R:3} must hold, and we show next that which particular case occurs, depends on the value of $R_0$.\\
{\bf Case 1}: $R_{0} < 1$.\\
In this case we will show that \pref{R:3} does hold for $F$. Here,  $\lambda= \rho(A_{0})<1$ because $R_0<1$, and thus 
$A^{t}_0\to 0$ as $t\to +\infty$.
 Let ${\mbi x}(t)$ be the solution of ~\eqref{model:1} with initial condition ${\mbi x}(0)>0$. We observe that $A({\mbi x})\leq A_0$ for all ${\mbi x}\geq 0$, where the matrix inequality is understood to hold entry-wise.  This implies that $\mbi{x}(t)\leq A^{t}_{\mbi{0}} \mbi{x}(0)\to 0$ as $t\to +\infty$, and therefore also $\mbi{x}(t) \to 0$  as $t\to +\infty$.  In summary, if $R_0<1$, then system ~\eqref{model:1} has a unique  fixed point at  the origin which attracts all orbits, by \pref{R:2} of Theorem \ref{CLST}.\\ 
{\bf Case 2}: $R_{0} >1$.\\
 In this case we will show that ~\pref{R:2} does not hold for $F$. Here, $\lambda=\rho(A_{0})>1$ because $R_0>1$. By the Perron-Frobenius Theorem, $A_0$ has a positive eigenvector $v$ corresponding to the eigenvalue $\lambda$. Set $\mbi{x}_{0}= \epsilon v$ for some small $\epsilon > 0$ to be determined momentarily. Using the Taylor expansion of $F$ near the origin, we have that 
$F(\mbi{x}_{0}) = \epsilon A_{0} v +O(\epsilon^2) = \epsilon \lambda v+O(\epsilon^2)$. Therefore, $F(\mbi{x}_{0})-\mbi{x}_{0} = \epsilon v (\lambda -1)+O(\epsilon^2) >> 0 $ for all sufficiently small $\epsilon>0$. Monotonicity of $F$ then implies that $\mbi{x}_{0} << F(\mbi{x}_{0}) \leq F^{2} (\mbi{x}_{0}) \leq  \cdots $, and hence the orbit of $\mbi{x}_{0}$ does not converge to $0$.  By Theorem \ref{CLST}, it follows that system \eqref{model:1} has  a unique positive fixed point, which attracts all nonzero orbits.\\

\end{proof}

\subsection{A coupled density-dependent model with a single egg stage}
We consider the following density-dependent coupled model of a n-stage migrant and a m-stage resident population:
 \begin{equation}
\label{model:2c}
\mbi{X}(t+1)= A_1({\mbi X}(t)){\mbi X}(t),
\end{equation}  
where
$A_1({\mbi{X}})=\left( \begin{array}{cccccccc}
                                      0 & f_{2}^s & \cdots & f_{n}^s&f_2^r &\cdots & f_{m-1}^r&f_m^r\\
                                      \phi t_{1}^s(x_{1}) & 0 & \cdots & 0&0&\cdots& 0&0 \\
                                      \vdots & \ddots  &  \vdots & \vdots &\vdots&\cdots&\vdots&\vdots\\
                                      0 &  \cdots  & t_{n-1}^s(x_{n-1}) & t_{n}^s(x_{n})&0&\cdots&0&0\\
                                      (1-\phi)t_1^r(x_{1})&0&\cdots&0&0&\cdots&0&0\\
                                      0&0&\cdots&0&t_2^r(x_{n+1})&\cdots & 0&0\\
                                      \vdots&\vdots&\cdots&\vdots&\vdots& \ddots&\vdots& \vdots \\
                                      0&0&\cdots&0&0&\cdots & t_{m-1}^r(x_{n+m-2})& t_m^r(x_{n+m-1})
                                      \end{array}\right)$,\\
and 
\begin{center}
$\mbi{X}=\left(
\begin{array}{c}
\text{eggs}(x_1)   \\
\text{first migrant stage} (x_2)  \\
\vdots \\
\text{last migrant stage}  (x_n)  \\
\text{first resident stage}  (x_{n+1})  \\
\vdots \\
\text{last resident stage}  (x_{n+m-1})  \\
\end{array}
\right)$
\end{center}

 Here, $\phi \in (0,1)$  represents the fraction of  eggs from the common egg pool that will enter the first migrant stage, and hence $1-\phi$ is the fraction that will enter the first resident stage. 
The fecundities $f_{i}^{s}$ and $f_{j}^{r}$ and transition probability functions $t_{1}^{s}(x_{1})$, $t_{i}^{s}(x_{i})$ and $t_{1}^{r}(x_{1}), t_{j}^{r}(x_{n+j-1})$ for $i= 2, \cdots, n$ and $j= 2, \cdots, m$ are assumed to satisfy the hypotheses ~\pref{A:1}-\pref{A:3} of two respective  models of the form ~\pref{model:1}, one for migrants and another for residents, for which we are using the functions $s^{s}_{i}(x)= x t_{i}^{s}(x)$  for $i=1, \cdots, n$ and $s^{r}_{j}(x)= x t_{j}^{r}(x)$ for $j=1, \cdots, m.$ To each model we also associate a basic reproduction number as in the previous subsection, and denote these as $R_0^s$ and $R_0^r$.\\

The Jacobian $A_{1,0}$ of system ~\eqref{model:2c} at the origin is obtained from $A_{1}({\mbi{X}})$ by replacing $t_{i}^s(x_{i})$ with $t_{i}^{s}(0)$ for $i=1, \cdots, n$ and  $t_{j}^s(x_{j})$ with $t_{j}^{r}(0)$ for $j=1, \cdots, m$.  The matrix $A_{1,0}$ is non-negative and irreducible, and we can associate a basic reproduction number to the linearized system, defined as  $R^{1c}_{0}=\rho({F^{1}_{0}(I-T^{1}_{0})^{-1}})$ where $F^{1}_{0}$ and $T^{1}_{0}$ are obtained from $A_{1,0}$ by setting respectively all $t_{i}^s=t_j^r=0$, and all $f_{i}^s=f_j^r=0$.  Note that in particular, the relationship $(\ref{convex})$ between the basic reproduction numbers 
associated to the coupled model, $R_0^{1c}$, and the two isolated models, $R_0^s$ and $R_0^r$, continues to hold in this context. Similarly, the sensitivity properties of $R_0^{1c}$ with respect  the parameters of the linearized system, and in particular  the remarks concerning the location of the maximum of $R_0^{1c}$, when considered as a function of $\phi$, as discussed in subsection 2.1, remain valid here as well.

We also have the following global stability result for the coupled system ~\eqref{model:2c}.

\begin{theorem}
 \label{p:2}
 Assume that (\ref{A:1})-(\ref{A:3}) hold for each of the two isolated population models. Assume also that
the maps $s_i^s$ and $s_j^r$ are strongly sublinear on ${\mbb R}_+$ for all $i=1,\dots,n$ and all $j=1,\dots,m$. 
If $R_{0}^{1c} < 1, $ then the origin is the unique fixed point of system ~\eqref{model:2c} and it is globally asymptotically stable. If $R_{0}^{1c} >1$, then system \eqref{model:2c} has a unique locally stable positive fixed point  which attracts all nonzero orbits.
\end{theorem}
\begin{proof}
We apply Theorem \ref{CLST}  to  system ~\pref{model:2c}. Let $G(\mbi{X})=A^{1}({\mbi{X}})\mbi{X}$ be obtained from  system ~\eqref{model:2c}.
Letting $G(\mbi{X})=(G_{1}(\mbi{X}), G_{2}(\mbi{X}), \cdots, G_{n+m-1}(\mbi{X}))$, we have more explicitly  that $G_{1}(\mbi{X})=\sum_{i=2}^{n}f^{s}_{i} x_{i}+ \sum_{i=2}^{m+n-1}f^{r}_{i}x_{i}$, \;\: $G_{2}(\mbi{X}) = \phi {s}_{1}^{s} (x_{1})$,  
  $G_{i}(\mbi{X})= s^{s}_{i-1}(x_{i-1})$ \;\:for \;\:$3 \leq i \leq n-1,  \;\: G_{n}(\mbi{X})= s^s_{n-1}(x_{n-1}) + s^s_{n}(x_{n}), \;\: G_{n+1}(\mbi{X})=(1-\phi_{s}) s_{1}^{r}(x_{1}), \;\: G_{n+j}(\mbi{X}) = s^{r}_{j}(x_{n+j-1})$   for  $2 \leq j \leq m-2$ and $G_{n+m-1}(\mbi{X})=s^r_{m-1}(x_{n+m-2})  + s^r_{m}(x_{n + m-1})$.

 All $G_k(\mbi{X})$ are either linear functions, or nonnegative linear combinations of strongly sublinear functions $s_l^s$ or $s_q^r$. 
 The rest of the proof is then similar to the proof of Theorem \ref{p:1}.
\end{proof}

\subsection{A coupled density-dependent model with two distinct egg stages}
A density-dependent coupled nonlinear model can be obtained from the linear model ~\eqref{model:2} in which the transition probability constants $t^{s}_{i}$ are replaced by functions $t^{s}_{i}(x_{i})$ for $i=1, \cdots, n$, and the constants $t^{r}_{j} $ are replaced by functions $t^{r}_{j}(x_{n+j})$ for $j=1, \cdots, m$. The coupled dynamics between an n-stage migrant and an m-stage resident population is then given by:
 \begin{equation}
\label{model:3}
\mbi{X}(t+1)= A_{2}({\mbi{X(t)}}) \mbi{X}(t),
\end{equation}\\
with
$A_{2}({\mbi{X}}) = \begin{pmatrix} 
 \phi_{s}*A_{n}^{s}(\mbi{X}) & (1-\phi_{r})* B_{n \times m}^{r}(\mbi{X})\\
 (1-\phi_{s})* B_{m \times n}^{s}(\mbi{X}) &  \phi_{r}*A_{m}^{r}(\mbi{X}) \\
 \end{pmatrix}. $\\

Here, the allocation parameters $\phi_s$ and $\phi_r$ belong to the interval $(0,1)$, and have the same interpretation as in subsection 2.2. Also, the submatrices in $A_2(\mbi{X})$ are defined in a similar way as in  model ~\eqref{model:2}, although here 
the transition probabilities are functions of the state rather than constants due to density dependence (fecundities remain constants). Again we assume that  all transition probability functions and fecundities satisfy the hypotheses ~\pref{A:1}-~\pref{A:3} as in two associated models of the form ~\pref{model:1}, one for migrants and one for residents, and where $s^{s}_{i}(x)= x \; t_{i}^{s}(x)$  for $i=1, \cdots, n$ and $s^{r}_{j}(x)= x \;t_{j}^{r}(x)$ for $j=1, \cdots, m.$ 
To each of these models we also associate a basic reproduction number based on the linearization at the origin, in the same way as in the previous subsection. We denote these as $R_0^s$ and $R_0^r$.

To the coupled model $(\ref{model:3})$, we also associate a basic reproduction number $R_0^{c2}$, based on the linearization at the origin. The relation between $R_0^{c2}$ on one hand, and $R_0^s$ and $R_0^r$ on the other, as discussed in subsection 2.2, remains valid here as well.

The global behavior of system $(\ref{model:3})$ is as follows:

\begin{theorem}  Assume that (\ref{A:1})-(\ref{A:3}) hold for each of the two isolated population models. Assume also that
the maps $s_i^s$ and $s_j^r$ are strongly sublinear on ${\mbb R}_+$ for all $i=1,\dots,n$ and all $j=1,\dots,m$. 
If  $R^{c2}_{0} <1$ then the fixed point at the origin is the only fixed point of system $(\ref{model:3})$, and it is globally asymptotically stable. If $R^{c2}_{0} >1$, 
then system $(\ref{model:3})$ has a unique positive locally stable fixed point  which attracts all nonzero orbits.
\end{theorem}
                                      
 \begin{proof}
Yet again, the proof of this theorem is an application of Theorem \ref{CLST} . Let $H(\mbi{X})=A_{2}({\mbi{X}})\mbi{X}$.
With $H(\mbi{X})=(H_{1}(\mbi{X}), H_{2}(\mbi{X}), \cdots, H_{m+n}(\mbi{X}))$, we have that $H_{1}(\mbi{X})=\phi_{s}(\sum_{i=2}^{n}f^{s}_{i} x_{i})+ (1-\phi_{r})(\sum_{i=n+2}^{2n}f^{r}_{i}x_{i})$,
$H_{i}(\mbi{X})=s^{s}_{i-1}(x_{i-1}) $, for $2 \leq i \leq n-1$, $H_{n}(\mbi{X})= s^{s}_{n-1}(x_{n-1}) + s^{s}_{n}(x_{n}),$ $H_{n+1}(\mbi{X})=(1-\phi_{s})(\sum_{i=2}^{n}f^{s}_{i} x_{i})+ \phi_{r}(\sum_{i=n+2}^{m+n}f^{r}_{i}x_{i})$, $H_{n+j}(\mbi{X})=s^{r}_{j}(x_{n+j}) $, for $ 2 \leq j \leq m-1,$  and  $H_{m+n}(\mbi{X})= s^{r}_{m-1}(x_{m+n-1}) + s^{r}_{m}(x_{m+n})$.

The coordinate functions $H_{i}(\mbi{X})$  of $H(\mbi{X})$ are either linear functions, or non-negative linear combinations of strongly sublinear functions  $s_l^s$ or $s_q^r$. The rest of the proof is then similar to the proof of Theorem \ref{p:1}.
\end{proof}

 \section{Conclusion} 
We developed four models of partial migration with the example of salmonid fishes in mind. The models address the cases of one pool of eggs (with a fraction becoming migrant and the rest resident) and two pools of eggs (one pool becomes migrant and the other resident), with and without density dependence. The evolutionary decision to become migratory or resident can take place at the egg stage or juvenile stage.  
We found that in the absence of density dependence that asymptotic growth in these models is governed by the value of the basic reproduction number. Under density dependence, the asymptotic dynamics are also governed by the value of a locally defined basic reproduction number, which determines whether the population goes extinct or settles at a unique fixed point consisting of a mixture of migrants and residents. 
Whether migration is determined at the egg or juveniles made no difference in the asymptotic dynamics between models. 
Our analysis shows that the basic reproduction number alone can be used to predict the stable coexistence or collapse of populations with partial migration.
One result of our analyses is that population persistence of both migrant and resident forms can occur with density dependence alone. 
However, our analysis did not consider this persistence in the context of evolution, in which new strategies that involve different allocation to resident and migrant forms can invade the population.
This question could be addressed with future research involving methods such as adaptive dynamics \cite{diekmann_beginners_2004}.

\section{Acknowledgments}
HAO and DAL acknowledge the U.S. Forest Service for partial support, PDL acknowledges partial support from NSF-DMS-1411853. We also thank an anonymous reviewer for helpful comments on an earlier version of the manuscript.

\bibliographystyle{siam}
\bibliography{density}

\end{document}